\documentclass[11pt]{amsart}
\usepackage{amsmath}
\usepackage{amsthm}
\usepackage{amscd}
\usepackage{amsfonts}
\usepackage{amssymb}
\usepackage{latexsym}

\usepackage{amsfonts}

\newtheorem{teo}{Theorem}

\newtheorem{lema}[teo]{Lemma}

\newtheorem{obs2}[teo]{Remark}
\newcommand{\pr}{\mathfrak{p}}
\newcommand{\Q}{\mathbb{Q}}
\newcommand{\Z}{\mathbb{Z}}
\newcommand{\F}{\mathbb{F}}
\newcommand{\N}{\mathbb{N}}
\newcommand{\ol}{\mathcal O}
\newcommand{\m}{\frak m}

\newcommand{\psqrt}{\;{}_{\pi}\!\!\sqrt}
\newcommand{\liminv}{\displaystyle \lim_{\leftarrow}}

\newcommand{\plim}[1]{\displaystyle{\lim_{\stackrel{\longleftarrow}{#1}}}\,}
\newcommand{\ilim}[1]{\displaystyle{\lim_{\stackrel{\longrightarrow}{#1}}}\,}
\newcommand{\dlog}{{\rm dlog\,}}

 \newcommand{\firstpageR}{1}
\begin{document}
\title[Reciprocity laws for local fields]{Reciprocity laws \`a la Iwasawa-Wiles}

\author{Francesc Bars, Ignazio Longhi}

\maketitle

\begin{center}
\begin{small}
\begin{abstract}
This paper is a brief survey on explicit reciprocity laws of
Artin-Hasse-Iwasawa-Wiles type for the Kummer pairing on local
fields.
\end{abstract}
\end{small}
\end{center}

\begin{section}{Introduction} Let $K$ be a complete discrete
valuation field, $\mathcal{O}_K$ its ring of integers,
$\mathfrak{m}_K$ its maximal ideal and $k_K$ the residue field.
Suppose that $K$ is an $\ell$-dimensional local field: this means
that there is a chain of fields $K_{\ell}=K, K_{\ell-1},\ldots, K_0$
where $K_{i}$ is a complete discrete valuation field with residue
field $K_{i-1}$ and $K_0$ is a finite field. We shall always assume
that $char(K_0)=p$.

Suppose $char(k_K)=p>0$: then we have the reciprocity law map
\begin{equation}\label{eq1}
(\;\, ,K^{ab}/K):K_{\ell}^M(K)\rightarrow Gal(K^{ab}/K),
\end{equation}
where $K^{ab}$ is the maximal abelian extension of $K$, $Gal$
denotes the Galois group and $K_{\ell}^M$ is the Milnor
$K$-theory.

Assume $char(K)=0$ and $\zeta_{p^m}\in K$, where $\zeta_{p^m}$ is a
primitive $p^m$-th root of unity. The classical Hilbert symbol
$(\;,\;)_m:K^*\times
K_{\ell}^{M}(K)\rightarrow<\zeta_{p^m}>=:\boldsymbol{\mu}_{p^m}$ is:
\begin{equation}\label{eq2}
(\alpha_0,\{\alpha_1,\ldots,\alpha_{\ell}\})_m:=\frac{\beta^{(\{\alpha_1,\ldots,\alpha_{\ell}\},K^{ab}/K)}}{\beta},
\end{equation}
where $\beta$ is a solution of $X^{p^m}=\alpha_0$ and
$(\{\alpha_1,\ldots,\alpha_{\ell}\},K^{ab}/K)$ is the element
given by the reciprocity law map.

When $\ell=1$, $K$ is the completion at some place of a global field
(i.e., a finite extension of $\mathbb{F}_p((T))$ or of the $p$-adic
field $\mathbb{Q}_p$), $K_1^M(K)\cong K^*$ and $(\;,K^{ab}/K)$ is
the classical norm symbol map of local class field theory.

Historically there was deep interest to compute this Hilbert
symbol (or, better, Kummer pairing) in terms of analytic objects,
as a step in the program of making local class field theory
completely explicit. Vostokov \cite{vos} suggests the existence of
two different branches of explicit reciprocity formulas: Kummer's
type and Artin-Hasse's type (later extended by Iwasawa and Wiles).
Kummer's reciprocity law \cite{kum} is
\begin{teo}[Kummer 1858] Let $K=\mathbb{Q}_p(\zeta_p)$, $p\neq 2$,
and $\alpha_0,\alpha_1$ principal units. Then
$$(\alpha_0,\alpha_1)_1=\zeta_p^{res(\log \tilde{\alpha_0}(X){\dlog}\tilde{\alpha_1}(X)X^{-p})}$$
where
$\tilde{\alpha_0}(X),\tilde{\alpha_1}(X)\in\mathbb{Z}_p[[X]]^*$
are power series such that $\tilde{\alpha_1}(\zeta_p-1)=\alpha_1$,
$\tilde{\alpha_0}(\zeta_p-1)=\alpha_0$, $res$ means the residue
and $\dlog$ is the logarithmic derivative.
\end{teo}
Artin-Hasse's reciprocity law \cite{artinhasse} is:
\begin{teo}[Artin-Hasse 1928]\label{theo2} Let $K=\mathbb{Q}_p(\zeta_{p^m})$,
$p\neq 2$, and $\alpha_1\in K^*$ a principal unit. Take
$\pi=\zeta_{p^m}-1$ a prime of $K$: then
$$(\zeta_{p^m},\alpha_1)_m=\zeta_{p^m}^{Tr_{K/\mathbb{Q}_p}(-\log \alpha_1)/p^m},\ \
(\pi,\alpha_1)_m=\zeta_{p^m}^{Tr_{K/\mathbb{Q}_p}(\pi^{-1}\zeta_{p^m}\log
\alpha_1)/p^m},$$ where $\log$ is the $p$-adic logarithm. Later
Iwasawa \cite{iwasawa} gave a formula for $(\alpha_0,\alpha_1)_m$
with $\alpha_0$ any principal unit such that
$val_{K}(\alpha_0-1)>\frac{2val_K(p)}{p-1}$.
\end{teo}
Roughly speaking the difference between these two branches is that
Kummer's type refers to residue formulas involving a power series
for each component of the pairing, while Artin-Hasse-Iwasawa's
type are non-residue formulas evaluating some generic series at
the $K$-theory component of the Hilbert pairing.

There is a big amount of articles in the literature that contribute
to and extend the above seminal works of Kummer and
Artin-Hasse-Iwasawa. The Hilbert symbol can be extended to
Lubin-Tate formal groups, and also to $p$-divisible groups. These
extensions are defined from Kummer theory: hence one often speaks of
Kummer pairing instead of Hilbert symbol. Wiles extended Iwasawa's
result to Lubin-Tate formal groups.

In this survey we intend to review some of the results on
Artin-Hasse-Iwasawa-Wiles' type reciprocity laws. We list different
variants of the Kummer pairing; afterwards we sketch some of the
main points in the proof of the Artin-Hasse-Iwasawa-Wiles
reciprocity law for 1-dimensional local fields. Finally we review
Kato's generalization of Wiles' reciprocity law, which is done in a
cohomological setting. Kato's work also extends the explicit
reciprocity law to higher dimensional local fields and to schemes.

For Kummer's type, one should cite also a lot of contributions by
Vostokov, and many others, for example Shafarevich, Kneser,
Br\"uckner, Henniart, Fesenko, Demchenko, and Kato and Kurihara in
the cohomological setting. We refer to Vostokov's paper \cite{vos}
for a list of results and references for reciprocity laws in this
case, adding to it the recent works of Benois \cite{ben} (for
cyclotomic extensions), Cao \cite{cao} (for Lubin-Tate formal
groups) and Fukaya \cite{fukaya} and the works of Ankeny and
Berrizbeitia \cite{ber}. Given $\alpha$ a unit of $\mathcal{O}_K$,
$K$ a totally ramified finite extension of $\mathbb{Q}_p$ and
$\pi$ a uniformizer of $K$, $\alpha$ has a factorization modulo
${K^*}^{p^n}$ by a product of $E(\pi^k)$, $k\in\N$, where $E$ is
the Artin-Hasse exponential. The contribution by Shafarevich is to
compute $(\alpha,\beta)_n$ in terms of the above factorization for
$\alpha$ and $\beta$. Berrizbeitia \cite{ber} recovers Shafarevich
results with different methods
 but also using the above factorization.

{\bf Notations}\\
Let $L$ be a valuation field: we denote $\mathcal{O}_L$ the ring of
integers, $\mathfrak{m}_L$ its maximal ideal, $k_L$ the residue
field and $val_L$ the valuation function. We write $\overline{L}$
for the algebraic closure of $L$ if $char(L)=0$ and the separable
closure otherwise; $\hat{\overline{L}}$ will be the completion
(which is algebraically closed).  Denote $Gal(\overline{L}/L)$ by
$G_L$ and the continuous Galois cohomology group $H^i(G_L,A)$ by
$H^i(L,A)$.

For any ring $R$, $R^*$ means the invertible elements of $R$. As
usual $\mathbb{F}_q$ is the finite field of cardinality $q$.

The symbol $M$ will always denote a finite extension of
$\mathbb{Q}_p$: $M_0$ is the subfield of $M$ such that $M/M_0$ is
totally ramified and $M_0/\Q_p$ is unramified.

For $N$ a $\Z_p$-module, $N(r)$ will be its $r$-th Tate twist,
$r\in\Z$ (recall that $\Z_p(1):=\plim{n}\boldsymbol{\mu}_{p^n}$)

\end{section}

\begin{section}{The Kummer pairing} Let $K$ be an
 $\ell$-dimensional local field. Then $K$ is isomorphic to one of
 \begin{itemize}
\item $\mathbb{F}_q((X_1))\ldots((X_{\ell}))$
if $char(K)=p$
\item $M((X_1))\ldots((X_{\ell-1}))$ if
$char(K_1)=0$
\item a finite\ extension\ of
$M\{\{T_1\}\}\ldots\{\{T_n\}\}((X_{n+2}))\ldots((X_{\ell}))$
 if\\
$char(K_{n+1})=0$ and $char(K_n)=p>0$ \end{itemize} where
$L\{\{T\}\}$ is
$$\left\{\sum_{-\infty}^{+\infty}a_iT^i:a_i\in L,
\inf(val_L(a_i))>-\infty,\displaystyle{\lim_{i\rightarrow
-\infty}}val_{L}(a_i)=+\infty \right\}.$$ Defining
$val_{L\{\{T\}\}}(\sum a_iT^i):=\min \{val_L(a_i)\}$ makes
$L\{\{T\}\}$ a discrete valuation field with residue field
$k_{L}((t))$ (see Zhukov \cite{zhu}).

Since we assume $char(k_K)=p>0$ it follows that the
$\ell$-dimensional local field $K$ has residue field isomorphic to
$\mathbb{F}_q((X_1))\ldots((X_{\ell-1}))$. Then we have the
reciprocity map (\ref{eq1}) obtained by Kato \cite{katoK} (see the
exposition \cite{kur2}), which had already been proved by Parshin
\cite{par} when $char(K)=p>0$.

Now we introduce the Kummer pairing through the classical Hilbert
symbol (\ref{eq2}). Assume first $char(K)=0$ and $\zeta_{p^m}\in K$
and restrict the Hilbert pairing (\ref{eq2}) to
$(1+\mathfrak{m}_K)\times K_{\ell}^{M}(K)$:  then
$$(1+\delta_0,\{\alpha_1,\ldots,\alpha_{\ell}\})_m=\frac{\beta^{(\{\alpha_1,\ldots,\alpha_{\ell}\}
,K^{ab}/K)}}{\beta}=\zeta_{p^m}^{c},$$ where $\beta$ is a solution
of $X^{p^m}=1+\delta_0$ and $c\in\mathbb{Z}/p^m\Z$ is determined by $\delta_0$ and
the $\alpha_i$'s. Recall that Milnor $K$-groups $K_{n}^M(K)$ are
defined as $(K^*)^{\otimes^n}$ modulo the subgroup generated by
$a\otimes (1-a)$, $a\in K^*-1$.

We rewrite the above restricted pairing as
$$\mathfrak{m}_K\times K_{\ell}^M(K)\rightarrow
W_{p^m}:=\{\zeta_{p^m}^j-1|j\in\mathbb{Z}\}=\{w\in\hat{\overline{K}}|(w+1)^{p^m}-1=0\}$$
\begin{equation}\label{eq3}
(\delta_0,\{\alpha_1,\ldots,\alpha_{\ell}\})_m:=\zeta_{p^m}^c-1.
\end{equation}
Recall that $\widehat{\mathbb{G}}_m$ is the formal group given by
$X+_{\widehat{\mathbb{G}}_m}Y:=X+Y+XY$ and
$-_{\widehat{\mathbb{G}}_m}X=\sum_{i=1}^{\infty}(-1)^iX^i$.
Summation $p^m$-times in $\widehat{\mathbb{G}}_m$ is given by
$[p^m](X)=(1+X)^{p^m}-1$. Observe that
\begin{equation}\label{eq4}
(\delta_0,\{\alpha_1,\ldots,\alpha_{\ell}\})_m=((\{\alpha_1,\ldots,\alpha_{\ell}\},K^{ab}/K)
-_{\widehat{\mathbb{G}}_m}1)(\beta)\end{equation} where $\beta$ is
a solution of $(1+X)^{p^{m}}-1=\delta_0$.

This reformulation of the classical Kummer pairing for
$\widehat{\mathbb{G}}_m$ can be extended to classical Lubin-Tate
formal groups, to Drinfeld modules and in greater generality to
$p$-divisible groups.

\begin{subsection}{Classical Lubin-Tate formal groups}
Lubin-Tate formal groups are defined for 1-dimensional local
fields: classically the emphasis is on the unequal characteristic
case and in this survey we shall restrict to such setting when
considering Lubin-Tate formal groups. We refer the interested
reader to \cite{haz} and \cite{serre} for proofs and detailed
explanations.

 Recall that a formal 1-dimensional commutative group law $F$
over $\mathcal{O}_M$ is $F(X,Y)\in\mathcal{O}_M[[X,Y]]$,
satisfying:\\ (i) $F(X,Y)\equiv X+Y \ mod\ deg\ 2$; \\(ii)
$F(X,F(Y,Z))=F(F(X,Y),Z)$; \\(iii) $F(X,Y)=F(Y,X)$.
%and (iv)
%$F(X,0)=X,F(0,Y)=Y$.
% and (v) exists the power series
%$-_{F}X:=\beta(X)\in\mathcal{O}_M[[X]]$, such that
%$F(X,\beta(X))=0$.

We fix $\pi$ a uniformizer of $M$. Define
$$\mathcal{F}_{\pi}:=\{f\in \mathcal{O}_M[[X]]\big{|}f\equiv\pi X\
mod\ deg\ 2, f\equiv X^{p^{l}}\ mod\ \mathfrak{m}_M\},$$ where
$p^{l}$ is the number of elements of $k_{M}$.

 \begin{teo}[Lubin-Tate] We have:\\
1. for every $f\in \mathcal{F}_{\pi}$ it exists a unique
1-dimensional commutative formal group law $F_f$ defined over
$\mathcal{O}_M$ and called of Lubin-Tate, such that $f\in
End(F_f)$ (i.e. $F_f(f(X),f(Y))=f(F_f(X,Y))$).\\ 2. Given
$f,g\in\mathcal{F}_{\pi}$ then $F_{g}$ and $F_f$ are isomorphic
over $\mathcal{O}_M$.\\ 3. Given $f\in \mathcal{F}_{\pi}$ and
$a\in\mathcal{O}_M$, it exists a unique
$[a]_f(X)\in\mathcal{O}_M[[X]]$ with $[a]_f\in End(F_f)$ and $
[a]_f(X)\equiv a X\ mod\ deg\ 2$; the map $a\mapsto [a]_f$ gives
an embedding $\mathcal{O}_M\rightarrow End(F_f)$. We denote also
$[a]_{F_f}$ by $[a]_f$.
\end{teo}
\noindent {\bf Example} The formal multiplicative group
$\widehat{\mathbb{G}}_m$ over $\Q_p$ corresponds to the Lubin-Tate
formal group with $\pi=p$ and $f=(1+X)^p-1\in\mathcal{F}_p$.\\

We denote by $F_f(B)$ the $B$-valued points of $F_f$. Consider

$$W_{F_{f},m}:=\{zeroes\ of\ [\pi^m]_{F_f}\ in\ F_f(\mathfrak{m}_{\hat{\overline{K}}})\}.$$

\begin{obs2}{\em
We want to emphasize that we have an ``embedding" of
$\mathcal{O}_M$ into $End(F_f)$, from which we get a tower of
field extensions $M(W_{F_f,m})$ in $\hat{\overline{M}}$. For
example take $F_f=\widehat{\mathbb{G}}_m$ and consider
$[p^m]_{\widehat{\mathbb{G}}_m}$ as $m$ varies in $\mathbb{N}$:
the roots of $[p^m]_{\widehat{\mathbb{G}}_m}(X)=(1+X)^{p^m}-1$ are
$\zeta_{p^m}-1$, thus $M(W_{\widehat{\mathbb{G}}_m,m})$ is the
cyclotomic extension $M(\zeta_{p^m})$. The groups $W_{F_f,m}$ are
the Lubin-Tate analogs of $\boldsymbol{\mu}_{p^m}$.}
\end{obs2}

Let $K$ be a finite extension of $M$ such that $W_{F_f,m}\subseteq
K$. The Kummer pairing for $F_f$ is:
$$(\;,\;)_m:F_f(\mathfrak{m}_K)\times K^*\rightarrow W_{F_f,m}$$
\begin{equation}\label{eq5} (a,u)_m:= ((u,K^{ab}/K)-_{F_f}1)\beta
\end{equation}
where $\beta$ is a solution of $[\pi^m]_{F_f}(X)=a$.

There is a more general notion of Lubin-Tate formal group, called
relative Lubin-Tate formal group, which involves the unique
unramified extension of $M$ of degree $d$. The corresponding
formulation of the Kummer pairing differs slightly from (\ref{eq5}).
We refer to \cite[Chapter 1, \S 1.1,\S1.4,\S4.1]{deshalit} for the
precise statement.

We also remark that there is a notion of $n$-dimensional
Lubin-Tate formal group \cite{haz}. Also in this case we have an
embedding of $\mathcal{O}_M$ into $End(F_f)$, and some generalized
Kummer pairing appears. Since Kummer pairing on $p$-divisible
groups will also include this case we do not discuss it any
further here.
\end{subsection}
\begin{subsection}{1-dimensional local fields in $char=p>0$: Drinfeld
modules} The key property of Lubin-Tate formal groups is the
embedding: $\mathcal{O}_M\rightarrow End(F_f)$. This suggests to
define a Kummer pairing for 1-dimensional local fields $K$ with
$char(K)=p>0$ in the following way.

For simplicity we take $F=\mathbb{F}_q(T)$ and put
$A=\mathbb{F}_q[T]$. Observe that
$End_{\mathbb{F}_q}({\mathbb{G}}_a/F)\cong F\{\tau\}$, where $\tau
a=a^q\tau$ for $a\in F$, $\tau^0=id$. We can think of $F\{\tau\}$ as
a subset of $F[X]$, via $\tau_0\leftrightarrow X,\tau\leftrightarrow
X^q$: then $F\{\tau\}$ consists of the additive polynomials and the
product in $F\{\tau\}$ corresponds to composition.

Drinfeld \cite{drinfeld} defined elliptic modules (now called
Drinfeld modules) as embeddings
$$\Phi: A\rightarrow End_{\mathbb{F}_q}({\mathbb{G}}_a/F)$$
$$a\mapsto \Phi_a=a+(\ldots)\tau ,$$ $\mathbb{F}_q$-linear and
non-trivial (i.e., there is $a\in A$ such that $\Phi_a$ is not
equal to $a$).

For $a\in A$ we write $\Phi[a]:=\{zeroes\ of\ \Phi_a(X)\}$. We
have that $\Phi[a]\cong (A/(a))^d$ and define $rank(\Phi):=d$. Let
$\pr=(\pi)$ be a place of $A$: under some technical assumptions
(see \cite{bl} for a reference) we can extend $\Phi$ to
$$A_{\pr}\rightarrow End_{\mathbb{F}_q}(\widehat{\mathbb{G}}_a)\cong F_{\pr}\{\{\tau\}\}$$
which we also call $\Phi$, $a\mapsto \Phi_a$, where $F_{\pr}$ is the
completion of the field $F$ at $\pr$ and $F_{\pr}\{\{\tau\}\}$ is
the ring of skew power series.

Denote by $W_{\Phi,\pi^m}$ the set of roots of $\Phi_{\pi^m}(X)$
in $\overline{F_{\pr}}$. For any finite extension $K/F_{\pr}$
containing $W_{\Phi,\pi^m}$ we define the Kummer pairing:

$$(\; ,\; )_m:\mathfrak{m}_K\times K^*\rightarrow W_{\Phi,\pi^m}$$
\begin{equation}\label{eq6}
(a,u)_m:=((u,K^{ab}/K)-1)(\beta)
\end{equation}
where $\beta$ is a root of $\Phi_{\pi}^m(X)=a$.

For example for rank 1, the simplest Drinfeld module is the
Carlitz module defined by $\Phi_{T}(\tau)=T id+\tau$. Then
$\Phi_{T^2}(\tau)=\Phi_{T}(\tau)\circ\Phi_{T}(\tau)=(T
id+\tau)\circ(T id+\tau)=T^2id+\tau
T+T\tau+\tau^2=T^2id+(T^q+T)\tau+\tau^2$ and similarly
$\Phi_{T^3}(\tau)=T^3id+(T^{2q}+T^{q+1}+T^2)\tau+(T^{q^2}+T^q+T)\tau^2+\tau^3$.
If we take $\pr=(T)$ ($\pi=T$) then
$F_{\pr}\cong\mathbb{F}_q((T))$ and $W_{\phi,T}=\{zeroes\ of\
TX+X^q\}$, $W_{\phi,T^2}=\{zeroes\ of\ T^2X+(T^q+T)X^q+X^{q^2}\}$,
$W_{\phi,T^3}=\{zeroes\ of\
T^3X+(T^{2q}+T^{q+1}+T^2)X^q+(T+T^q+T^{q^2})X^{q^2}+X^{q^3}\}$,
...
\end{subsection}
\begin{subsection}{$p$-divisible groups}

A vast extension of the theory above considers $p$-divisible
groups (also called Barsotti-Tate groups). For the definition and
main properties see \cite{tate}.

Let $G$ be a $p$-divisible group scheme over $\mathcal{O}_K$ of
dimension $d$ and finite height $h$, where $K$ is any
$\ell$-dimensional local field with $char(k_K)=p>0$. We denote by
$[p^m]_G$ the $p^m$-th power map $G\rightarrow G$ and let $G[p^m]$
be the group scheme kernel. As usual, $\mathfrak{X}(B)$ denotes
the $B$-points of the scheme $\mathfrak{X}$. We put
$W_{G,p^m}:=G[p^m](\mathcal{O}_{\overline{K}})$ and we impose
$W_{G,p^m}=G[p^m](\mathcal{O}_K)$, i.e. $W_{G,p^m}\subset K$. Then
Kummer pairing is defined by:

$$(\;,\;)_m:G(\mathfrak{m}_K)\times K_{\ell}^M(K)\rightarrow W_{G,p^m}$$
\begin{equation}\label{eq7}
(a,u)_m:= ((u,K^{ab}/K)-_{G}1)(\beta)
\end{equation} where $\beta$ is a root of $[p^m]_G(X)=a$. Finally
we observe that $p$-groups with an action of $\mathcal{O}_K$ have
also been studied \cite{tate}. In this case (\ref{eq7}) can be
reformulated with respect to $[\pi^m]_G$, $\pi$ a uniformizer of
$K$.
\end{subsection}

\begin{subsection}{Cohomological interpretation}
Let $K$ be an $\ell$-dimensional local field with $char(k_K)=p>0$
and $char(K)=0$ . The Kummer pairing also admits an interpretation
in terms of Galois cohomology. We do this in the generality of
$p$-divisible groups. Consider the exact sequence
$$0\rightarrow G[p^m](\mathfrak{m}_{\overline{K}})\rightarrow
G(\mathfrak{m}_{\overline{K}})\rightarrow
G(\mathfrak{m}_{\overline{K}})\rightarrow 0.$$ This induces:
$$\delta_{1,G,m}:G(\mathfrak{m}_K)\rightarrow
H^1(K,G[p^m](\mathcal{O}_{\overline{K}})).$$ We assume
$W_{G,p^m}\subset K$, i.e.
$G[p^m](\mathcal{O}_{\overline{K}})=G[p^m](\mathcal{O}_{{K}})$.
The Galois symbol map:
$$h^r_K:K_{r}^M(K)\rightarrow H^{r}(K,\mathbb{Z}_p(r))$$
is obtained by cup product $h^1_K\cup\ldots\cup h^1_K$, where
$h^1_K$ is the connecting morphism $h^1_K:K^*\rightarrow
H^1(K,\mathbb{Z}_p(1))$ from the usual Kummer sequence.

There is a canonical isomorphism \cite{katoK}
$H^{\ell+1}(K,\mathbb{Z}_p(\ell))\cong\mathbb{Z}_p$ defining the
following pairing:
$$\widehat{(\;,\;)}_m:G(\mathfrak{m}_K)\times K_{\ell}^M(K)\rightarrow
H^{\ell+1}(K,\mathbb{Z}_p(\ell)\otimes
G[p^m](\mathcal{O}_{K}))\cong G[p^m](\mathcal{O}_K)$$
\begin{equation}\label{eq8}
\widehat{(a,u)}_m:=(-1)^{\ell} \delta_{1,G,m}(u)\cup
h^{\ell}_K(a).
\end{equation}
Then by \cite[Proposition 6.1.1]{fukaya}
$\widehat{(a,u)}_m=((u,K^{ab}/K)-_{G}1)(\beta)=(a,u)_m$ where
$\beta$ is a root of $[p^m]_G(X)=a$.

If one could define a good analog in characteristic $p$ of the
Galois symbol maps for the cohomological groups from Illusie
cohomologies used by Kato \cite{katoK} to obtain the reciprocity
law map $(\;,K^{ab}/K)$, then a cohomological interpretation
should appear when $char(K)=p>0$ (following arguments like the
proof of \cite[Proposition 6.1.1]{fukaya}). We refer to \cite[\S
5]{kur2} for a very quick review of the definitions of these
cohomological groups and of Kato's higher local class field
theory.
\end{subsection}
\begin{subsection}{Limit forms of the Kummer pairing}
In this subsection $(\;,\;)_m$ denotes any of the pairings
(\ref{eq4}) to (\ref{eq7}).

Let $W_{\sharp,m}$ be one of $W_{F_f,\pi^{m}}$, $W_{\Phi,\pi^m}$
or $W_{G,p^m}$. We shorten $K(W_{\sharp,m})$ to $K_{\sharp,m}$ and
write $\mathcal{O}_{\sharp,m}$ (resp. $\mathfrak{m}_{\sharp,m}$)
for the ring of integers (resp. the maximal ideal). By definition
the Tate module is $T_p(\sharp):=\plim{m}W_{\sharp,m}$ where the
limit is taken with respect to the map $[\sharp]$ (which means
$[\pi]_f,\Phi_{\pi}$ or $[p]_G$). Consider
$W_{\sharp,\infty}:=\ilim{m}W_{\sharp,m}=\cup_mW_{\sharp,m}$.

 The symbol
$\mathcal{M}_{\sharp,m}$ denotes one of
$F_f(\mathfrak{m}_{\sharp,m})$, $\mathfrak{m}_{\sharp,m}$ or
$G(\mathfrak{m}_{\sharp,m})$. Consider
$\ilim{m}\mathcal{M}_{\sharp,m}$ as the direct limit of $[\sharp]$
: it consists of sequences $(a_n)_{n\geq N}$ (for some
$N\in\mathbb{N}$) such that $a_n\in K_{\sharp,n}$ and
$a_{n+1}=[\sharp]a_n$.

We need to impose that the $K_{\sharp,m}$'s are {\bf abelian}
extensions of $K$ (this is satisfied for example by 1-dimensional
classical Lubin-Tate formal groups and rank 1 Drinfeld modules).
Then we have a limit version of the Kummer pairing:
\begin{equation}\label{eq9}
(\;,\;):\ilim{m}\mathcal{M}_{\sharp,m}\times
\plim{m}K_{\ell}^M(K_{\sharp,m})\rightarrow W_{\sharp,\infty}
\end{equation}
where $\plim{m}K_{\ell}^M(K_{\sharp,m})$ is with respect to the
Norm map. The limit pairing is well defined: by the abelian
assumption we have (Kato-Parshin's reciprocity law)
$(N_{K_{\sharp,n}/K_{\sharp,n-1}}(u'),K_{\sharp,n-1}^{ab}/K_{\sharp,n-1})=(u',K_{\sharp,n}^{ab}/K_{\sharp,n})$
acting on the roots of $[\sharp]^{n-1}(X)=a$, when $u'\in
K_{\ell}^M(K_{\sharp,n})$.

We remark that the above is often formulated without taking the
whole limit tower. That is, suppose that one wants to compute
explicitly $(a,u)_m$ and $u'\in K_{\ell}^M(K_{\sharp,m+k})$ is
given so that $N_{K_{\sharp,m+k}/K_{\sharp,m}}(u')=u$: then by a
similar argument as needed for defining (\ref{eq9}) we have
\begin{equation}\label{eq9,5}
(a,u)_m=([\sharp]a,N_{K_{\sharp,m+k}/K_{\sharp,m+1}}(u'))_{m+1}=\ldots=([\sharp]^k
a,u')_{m+k}.\end{equation}

We can also write the above limit form (\ref{eq9}) as:
\begin{equation}\label{eq10}
\plim{n}K_{\ell}^M(K_{\sharp,n})\rightarrow
Hom(\mathcal{M}_{\sharp,m},T_p(\sharp))
\end{equation} sending $u=(u_n)_n$ to $\displaystyle
a\mapsto \lim_{n\rightarrow\infty}(a,u_n)_n$ with fixed $m$. These homomorphisms are
continuous because of the continuity of the Kummer pairing: see
\cite[Lemma 15]{bl} for the case $\ell=1$.
\end{subsection}
\end{section}

\begin{section}{Explicit reciprocity law formulas \`a la Wiles' for 1-dimensional local fields}

In this section we fix $\ell=1$ and sketch some of the main ideas to
obtain explicit reciprocity laws for a 1-di\-men\-sio\-nal local
field $K$ in the context of classical Lubin-Tate formal groups and
rank 1 Drinfeld modules, introducing Coleman power series. See the
last section for different approaches to this explicit reciprocity
law through the  exponential or dual exponential map.

Let $K$ be either $M$ or $F_{\pr}$. In the Drinfeld module case
$\Phi$ is required to be of rank 1 (in order to obtain abelian
extensions) and sign-normalized. We refrain from explaining this
last technical condition (the reader is referred to \cite[\S2.1]{bl}
and the sources cited there) and just notice that when
$A=\mathbb{F}_q[T]$ as in \S 2.2 this means that $\Phi$ is the
Carlitz module, i.e. $\Phi_T(X)=TX+X^q$. Furthermore we take $\pi\in
A$ a monic polynomial.

We lighten the notation introduced in \S 2.5 by shortening
$K_{\sharp,m}$ to $K_m$.

The extensions $K_m/K$, are totally ramified and abelian: they are
generated by roots of Eisenstein polynomials and
$Gal(K_{m}/K)\cong (\mathcal{O}_K/(\pi)^n)^*$. The Tate module
$T_p(\sharp)$ is a rank 1 $\mathcal{O}_K$-module with
$\mathcal{O}_K$-action given by $\alpha\cdot
\gamma:=[\alpha]_f\gamma$ or $\Phi_{\alpha}(\gamma)$. Let
$(\varepsilon_n)_n$ be an $\mathcal{O}_K$-generator of
$T_p(\sharp)$: then $K_{\sharp,m}=K(\varepsilon_m)$, since
$\varepsilon_m$ generates $W_{\sharp,m}$, and one has
$[\sharp]\varepsilon_{n+1}=\varepsilon_n$,
$[\sharp]^{n+1}\varepsilon_n=0$ where $[\sharp]$ is respectively
$[\pi]_f$ or $\Phi_{\pi}$. Moreover the $\varepsilon_n$'s form a
norm compatible system. Denote $\cup_m K(W_{\sharp,m})$ by
$K_{\sharp,\infty}$.

In this section $(\;,\;)_m$ refers to (\ref{eq5}) or (\ref{eq6}) and
$(\;,\;)$ refers to (\ref{eq9}).

First, notice that $(\;,\;)_m$ is bilinear, additive in the first
variable and multiplicative in the second variable. In particular
$(\;,\zeta)_m=0$ for any root of unity $\zeta$.

Consider the character
$$\chi: Gal(K_{\sharp,\infty}/K)\rightarrow \mathcal{O}_K^*$$
$\sigma\mapsto \chi_{\sigma}$ defined by
$\sigma((\varepsilon_n)_n)=\chi_{\sigma}\cdot((\varepsilon_n)_n)$;
it can be thought of as a character of $Gal(K^{ab}/K)$. We remark
that for $\widehat{\mathbb{G}}_m$ this $\chi$ corresponds to the
cyclotomic character $\chi_{cycl}$. The main point is the
following: $\chi^{-1}:\mathcal{O}_K^*\rightarrow
Gal(K_{\sharp,\infty}/K)$ coincides with the inverse of the local
norm symbol map. Therefore:
\begin{equation}
\label{eq55} (\varepsilon_m,u)_m=\left\{\begin{array}{c}
([N_{K_{m}/K}u]_f^{-1}-_{F_f}1)(\varepsilon_{2m})\ \ K=M\\
\Phi_{(N_{K_{m}/K}u)^{-1}}(\varepsilon_{2m})-\varepsilon_{2m}\hspace{35pt} K=F_{\pr}.\\
\end{array}\right.\end{equation}

Inspired by (\ref{eq55}) one can ask: could we find $h$ such that
$(a_m,u_m)_m=[h(a,u)]_f(\varepsilon_m)$ for classical Lubin-Tate
formal groups or $=\Phi_{h(a,u)}(\varepsilon_m)$ for rank 1
Drinfeld modules? Observe that $h(a,u)$ modulo $\pi^{m}$ defines
the same action over $\varepsilon_m$ (because $W_{\sharp,m}$ is
isomorphic to $\mathcal{O}_K/(\pi^m)$).

In this direction one obtains the following results. From class
field theory it follows $(a,a)_n=0$. Exploiting linearity and
continuity of the pairing $(\;,\;)_m$ one obtains (assuming $v(c)$
greater than a fixed value which depends linearly of $n$),
\begin{equation}\label{eq56} (c,w)_n=(c\varepsilon_n\frac{\dlog
w}{{\rm d}\varepsilon_n},\varepsilon_n)_n
\end{equation}
 where
$$\dlog:\mathcal{O}_{\sharp,n}^*\rightarrow
\Omega^1_{\mathcal{O}_{\sharp,n}/\mathcal{O}_K}$$ is the map
$x\mapsto \frac{dx}{x}$ (recall that the module of K\"ahler
differentials $\Omega^1_{\mathcal{O}_{\sharp,n}/\mathcal{O}_K}$ is
free with generator ${\rm d}\varepsilon_n$ over
$\mathcal{O}_{\sharp,n}/\mathfrak{d}_{K_{n}/K}\,$, where
$\mathfrak{d}_{K_{n}/K}$ is the different of $K_{n}$ over $K$).
Notice that one can pick a power series $g\in\mathcal{O}_K((X))$
such that $g(\varepsilon_n)=w$ and $\dlog w/{\rm
d}\varepsilon_n=g'(\varepsilon_n)/g(\varepsilon_n)$.

Finally one proves that exists $m>n$ such that
$$(a_n,\varepsilon_n)_n=-(\varepsilon_m,1+a_m\varepsilon_m^{-1})_m$$
giving a positive answer to the question above.

This suggests the following definition of an analytic pairing:
$$[a,u]_n:=Tr_{K_{n}/K}\left(\pi^{-n}\lambda(a)\frac{\dlog u}{{\rm d}\varepsilon_n}\right)\cdot\varepsilon_n\,.$$
Here $\cdot$ is the action on the Tate module at level $n$ and
$\lambda$ is the logarithm map defined by $\lambda(a):=\lim
\frac{1}{\pi^n}[\sharp]^na$ (the limit exists for $val_K(a)$
sufficiently big).

In order to compare $(a,u)_n$ and $[a,u]_n$, Iwasawa (theorem
\ref{theo2}) imposes the condition that there exists $m$ such that
$u=N_{\Q_p(\zeta_{p^m})/\Q_p(\zeta_{p^n})}(u')\,$: then
$(a,u)_n=([p]_{\widehat{\mathbb{G}}_m}^{m-n}a,u')_m$ and it is at
level $m$ that he compares the two pairings obtaining
$([p]_{\widehat{\mathbb{G}}_m}^{m-n}a,u')_m=[[p]_{\widehat{\mathbb{G}}_m}^{m-n}a,u']_m$.

The general case follows Iwasawa's argument. Here we will state a
limit version, hence we suppose that $(u_n)\in \liminv K_{n}^*$
(limit w.r.t.\! the norm).

To express in compact form the limit of the pairings $[\;,\;]_n$ it
is convenient to introduce Coleman's power series.

\begin{teo} Let $K$, $F_f$, $\Phi$ be as above. Then
\begin{enumerate}
\item \cite{coleman} (case $K=M$) There exists a unique operator
$\mathcal{N}$ (the Coleman norm) defined by the property
$$(\mathcal{N}h)\circ f(X)=\prod_{w\in W_{F_f,0}}h(X+_{F_f}w)$$ for
any $h\in M((x))_1$ (the set of those Laurent series which are
convergent in the unit ball).
\item \cite{bl} (case $K=F_{\pr}$) There exists a unique operator
$\mathcal{N}$ such that
$$(\mathcal{N}h)\circ\Phi_{\pi}=\prod_{v\in\Phi[\pr]}h(x+v)$$
for any $h\in F_{\pr}((x))_1$.\end{enumerate} Moreover, in both
cases the evaluation map $ev:f\mapsto (f(\varepsilon_{n}))_{n}$
gives an isomorphism
$$(\mathcal{O}_K((x))^*)^{\mathcal{N}=id}\cong \lim_{\leftarrow}K_{\sharp,n}^*\,.$$
\end{teo}
Denote by $Col_u$ the power series associated to
$u\in\displaystyle\lim_{\leftarrow}K_{\sharp,n}^*$. Then we define
the limit form of the analytic pairing by
$$[\;,\;]:\ilim{n}\mathcal{M}_{\sharp,m}\times \plim{n}K_{\sharp,m}^*\rightarrow W_{\sharp,\infty}$$
$$[a,u]:=Tr_{K_{\sharp,n}/K}(\pi^{-n}\lambda(a_n)\dlog Col_u(\varepsilon_n))\cdot\varepsilon_n\,.$$

Notice that $\dlog Col_u(\varepsilon_n)=\frac{\dlog
u_n}{d\varepsilon_n}$.

Finally one observes that $[a,u]$ has similar properties to
$(a,u)$; in particular
$[a,u]=[a_n\varepsilon_n\dlog Col_u(\varepsilon_n),\varepsilon_n]_n$
(compare with (\ref{eq56})) and to prove the reciprocity law one
is reduced to show $[\;,\varepsilon_n]_n=(\;,\varepsilon_n)_n$ for
$n$ big enough.

\begin{teo} \label{theor1}
Under the above notation, we have
\begin{enumerate} \item \cite{wiles} $(\;,\;)=[\;,\;]$ for
classical Lubin-Tate formal groups $F_f$. \item \cite{bl}
$(\;,\;)=[\;,\;]$ for $F_{\pr}$ and $\Phi$.
\end{enumerate}
\end{teo}

In \cite{wiles} Wiles proved this result for classical Lubin-Tate
formal groups without using Coleman power series: he takes $m$ big
enough in order to compare $(\;,\;)_m$ and $[\;,\;]_m$, like Iwasawa
had done for $\widehat{\mathbb{G}}_m$. This strategy requires a very
precise valuation calculation. For a detailed explanation one can
look also at \cite[\S 8,9]{lang}. The above limit version with
Coleman power series for classical Lubin-Tate formal groups can be
found in \cite[I,\S 4]{deshalit}.

For the Carlitz module, Angl\`es in \cite{angles} obtained Wiles'
version of theorem \ref{theor1}, i.e. without Coleman power
series.

\end{section}

\begin{section}{Explicit reciprocity laws and higher $K$-theory groups}
The starting point for the results in \S3 is explicit local
class field theory applied to $\varepsilon$ and the $\dlog$
homomorphism. In \cite{kato1} and \cite{kato} Kato reinterpreted
Wiles' reciprocity law as an equality between two maps obtained by
composition of natural maps from cohomology groups and gave
generalized reciprocity laws in higher $K$-theory. Here we
introduce his approach to the classical Hilbert symbol for
$\ell$-dimensional local fields, $\ell>1$, and reformulate parts
of it for rank 1 Drinfeld modules.

\begin{subsection}{Exponential map in the Hilbert symbol}
In this paragraph we assume $char(K)=0$ and $char(k_K)=p>0$. Take:
$$exp_{\delta}:\mathcal{O}_K\rightarrow\mathcal{O}_K^*$$ given by
$exp_{\delta}(a)=exp(\delta a)$ if
$val_K(\delta)>\frac{val_K(p)}{(p-1)}$ where
$exp(X)=\sum_{m\geq0}\frac{X^m}{m!}$ is the exponential. Assuming
$val_K(\delta)\geq \frac{2val_K(p)}{(p-1)}$, the map $exp_{\delta}$
extends to a group morphism \cite{kurihara}:
$$exp_{\delta,r}:\Omega_{\mathcal{O}_K}^{r-1}\rightarrow
\widehat{K_r^M(K)}$$
$$a \dlog b_1\wedge\cdots\wedge \dlog b_{r-1}\mapsto \{exp(\delta
a),b_1,\ldots,b_{r-1}\},$$ where
$\widehat{K_r^M(K)}:=\plim{n}K_r^M(K)/p^nK_r^M(K)$,
$\Omega_R^r:=\bigwedge_R^r\Omega^1_R$, $\Omega^1_R$ are the K\"ahler
differentials and $r$ is a strict positive integer. The function
$exp_{\delta,r}$ factors through
$\bigwedge_{\mathcal{O}_K}^{r-1}\hat{\Omega}^1_{\mathcal{O}_K}$
where $\hat{\Omega}^1_{\mathcal{O}_K}$ is the $p$-adic completion of
$\Omega^1_{\mathcal{O}_K}$.

Sen \cite{sen} generalizes theorem \ref{theo2}:
\begin{teo}\label{theosen} Assume $\zeta_{p^n}\in M$. Fix
$\pi$ a uniformizer of $M$, $\alpha_1\in\mathcal{O}_M-\{0\}$, and
$\alpha\in K$ with $val_K(\alpha-1)\geq
\frac{2val_K(p)}{p-1}+val_K(\alpha_1)$. Then
$$(\alpha,\alpha_1)_n=\zeta_{p^n}^c,\text{ with }
c=\frac{-1}{p^n}Tr_{M/\Q_p}(\frac{\zeta_{p^m}}{h'(\pi)}\frac{g'(\pi)}{\alpha_1}\log\alpha)$$
where $g,h\in\mathcal{O}_{M_0}[T]$ are such that $g(\pi)=\alpha_1$
and $h(\pi)=\zeta_{p^n}$, and the pairing is (\ref{eq2}).
\end{teo}
From now on take $val_M(\eta)=\frac{2val_M(p)}{p-1}$, $\eta\in
M_0(\zeta_{p^n})$. The proof of theorem \ref{theosen} reduces to
theorem \ref{theo2} because
$Tr_{M/M_0(\zeta_{p^m})}(ad\pi)=Tr_{M/M_0(\zeta_{p^m})}(\frac{a}{h'(\pi)})d\zeta_{p^m}$
and the commutativity of the following diagram
(\cite[p.217]{kurihara}):
\[ \begin{CD}
\hat{\Omega}^1_{\mathcal{O}_{M}} @>exp_{\eta,2}>>\widehat{K_2^M}(M) @>h_M>> \Z/p^n(1)\\
 @VV{Tr_{M/M_0(\zeta_{p^m})}}V     @VV{N_{M/M_0(\zeta_{p^m})}}V @VV{id}V\\
\hat{\Omega}^1_{\mathcal{O}_{M_0(\zeta_{p^n})}} @>exp_{\eta,2}>> \widehat K_2^M(M_0(\zeta_{p^n})) @>h_{M_0(\zeta_{p^n})}>> \Z/p^n(1)\\
\end{CD} \]\\
where $h_*$ is the Hilbert symbol $\{a,b\}\mapsto(a,b)_n$.

We rewrite Sen's theorem as a commutative diagram. Take
$\gamma\in\mathcal{O}_M$ with $val_M(\gamma)=\frac{val_M(p)}{p-1}$.
Fontaine proved that we have the following isomorphism
$$\Psi_M:\hat{\Omega}^1_{\mathcal{O}_{M}}/p^n\gamma^{-1}\rightarrow
\gamma^{-1}\mathfrak{d}^{-1}_{M/M_0}/\gamma^{-1}\mathfrak{d}^{-1}_{M/M_0(\zeta_{p^n})}(1)$$
induced from the map $a\dlog\zeta_{p^m}\mapsto
 ap^{-m}\otimes(\zeta_{p^m})_{m>0}$ where $\mathfrak{d}_{L_1/L_2}$ is the different of the finite
 extension of fields $L_1/L_2$. Commutativity of the
 diagram
\[ \begin{CD}
\hat{\Omega}^1_{\mathcal{O}_{M}}/p^n\gamma^{-1} @>\Psi_M>>\gamma^{-1}\mathfrak{d}^{-1}_{M/M_0}/
\gamma^{-1}\mathfrak{d}^{-1}_{M/M_0(\zeta_{p^n})}(1)\\
 @V{-exp_{\eta}}VV     @V{Tr_{\eta}\otimes id}VV\\
  K_2^M(M)/p^n @>h_M>> \Z/p^n(1)\\
 \end{CD} \]
(where $Tr_{\eta}$
 is the map induced by $x\mapsto Tr_{M/\Q_p}(\eta x)$) is Sen's theorem \cite[\S4]{kurihara}.

Now we consider $\ell>1$. Assume $\zeta_{p^n}\in K$. Consider
$K_0$ with $K/K_0$ a finite and totally ramified extension such that
$p$ is a prime element of $\mathcal{O}_{K_0}$. Recall
$k_{K}=\F_{q}((t_1))\ldots((t_{\ell-1}))$. Take $M$ such that
$M=M_0$, $M\subseteq K_0$ and the residue field is $\F_q$: then
$K_0=M\{\{T_1\}\}\cdots\{\{T_{\ell-1}\}\}$. Kurihara extends Sen's
result to higher Milnor $K$-theory (theorem \ref{theokurih}) by
means of the commutativity of the diagram:
\[ \begin{CD}
\hat{\Omega}^{\ell}_{\mathcal{O}_{K}} @>exp_{\eta,\ell+1}>>\widehat{K_{\ell+1}^M}(K) @>h_K>> \Z/p^n(1)=\boldsymbol{\mu}_{p^n}\\
 @VV{Tr_{K/K_0(\zeta_{p^m})}}V     @VV{N_{K/K_0(\zeta_{p^m})}}V @VV{id}V\\
\hat{\Omega}^{\ell}_{\mathcal{O}_{K_0(\zeta_{p^n})}} @>exp_{\eta,\ell+1}>> \widehat{K_{\ell+1}^M}(K_0(\zeta_{p^n})) @>h_{K_0(\zeta_{p^n})}>> \Z/p^n(1)\\
@VV{Res}V     @VV{\underline{Res}}V @VV{id}V\\
\hat{\Omega}^1_{\mathcal{O}_{M(\zeta_{p^n})}} @>exp_{\eta,2}>> \widehat{K_2^M}(M(\zeta_{p^n})) @>h_{M(\zeta_{p^n})}>> \Z/p^n(1)\\
\end{CD} \]\\
where
$Res:\hat{\Omega}^{\ell}_{\mathcal{O}_{K_0(\zeta_{p^n})}}\rightarrow
\hat{\Omega}^1_{\mathcal{O}_{M}}$ is defined by $Res(w \dlog
T_1\wedge\cdots\wedge\dlog T_{\ell-1})=w$ where
$w\in\hat{\Omega}^1_{\mathcal{O}_{M(\zeta_{p^n})}}$,
$\underline{Res}$ is the Kato's residue homomorphism in Milnor
$K$-groups and $h_L$ is the Hilbert symbol
$\{a_1,\ldots,a_{\ell+1}\}\mapsto
(a_1,\{a_2,\ldots,a_{\ell+1}\})_n$.
\begin{teo}[\cite{kurihara}]\label{theokurih}
Take $\alpha_1,\ldots,\alpha_{\ell}\in\mathcal{O}_K^*$ and
$\alpha\in\mathcal{O}_K^*$ with
$val_K(\alpha-1)\geq\frac{2val_K(p)}{p-1}$. Take
$f_i(T,T_2,\ldots,T_{\ell})\in\mathcal{O}_{K_0}[T]$ such that
$f_i(\pi,T_2,\ldots,T_{\ell})=\alpha_i$, and
$h(T)\in\mathcal{O}_{K_0}[T]$ with $h(\pi)=\zeta_{p^n}$ where $\pi$
is a fixed uniformizer of $K$. Then
$(\alpha,\{\alpha_1,\ldots,\alpha_{\ell}\})=\zeta_{p^n}^c$ where
%{\tiny $$c=-\frac{1}{p^n}Tr_{M(\zeta_{p^n})/\Q_p}\circ
%c_{K_0(\zeta_{p^n})/M(\zeta_{p^n})}\circ
%Tr_{K/K_0(\zeta_{p^n})}\left(\log \alpha\frac{T_2\cdots
%T_{\ell}}{\alpha_1\cdots\alpha_{\ell}}\frac{\zeta_{p^n}}{h'(\pi)}\left[
%\det(\frac{\partial f_i}{\partial T_j})_{1\leq
%i,j\leq\ell}\right]_{|T_1=\pi}\right),$$}
$$c=-\frac{1}{p^n}\mathcal T\left(\log \alpha\frac{T_2\cdots
T_{\ell}}{\alpha_1\cdots\alpha_{\ell}}\frac{\zeta_{p^n}}{h'(\pi)}\left[
\det(\frac{\partial f_i}{\partial T_j})_{1\leq
i,j\leq\ell}\right]_{|T_1=\pi}\right)\,.$$ Here $\mathcal
T:=Tr_{M(\zeta_{p^n})/\Q_p}\circ
c_{K_0(\zeta_{p^n})/M(\zeta_{p^n})}\circ Tr_{K/K_0(\zeta_{p^n})}$,
with $c_{L\{\{T\}\}/L}(\sum a_iT^i):=a_0$ and
$c_{L\{\{T_2\}\}\{\{T_3\}\}\cdots\{\{T_k\}\}/L}$ defined recursively
as composition of the maps
$c_{L\{\{T_2\}\}\{\{T_3\}\}\cdots\{\{T_{i+1}\}\}/L\{\{T_2\}\}\{\{T_3\}\}\cdots\{\{T_{i}\}\}}$.
\end{teo}
\begin{obs2}{\em
Benois in \cite{benois} extends theorem \ref{theosen} to formal
groups.}
\end{obs2}
\end{subsection}

\begin{subsection}{Kato's generalized explicit reciprocity laws}
Kato generalizes Wiles' reciprocity law for an unequal
characteristic local field $K$ giving a natural interpretation in
the context of cohomology groups. Here we introduce two
generalizations.
\begin{subsubsection}{Local approach}
Let us first rewrite (\ref{eq10}) in the context of Wiles'
reciprocity law: let $\pi$ be a fixed uniformizer of $M$ , $F_f$ a
formal Lubin-Tate group and $\varepsilon$ a fixed generator for
the Tate module. For simplicity we assume that $p$ is prime in
$K$. Recall \cite[Remark 4.1.3]{kato} that a Lubin-Tate formal
group $G$ over $\mathcal{O}_M$ has the following characterization
in $p$-divisible groups: $dim(G)=1$, and the canonical map
$End(G)\rightarrow End_{\mathcal{O}_M}(Lie G)\cong \mathcal{O}_M$
is an isomorphism, where $Lie(G)$ is the tangent of $G$ at the
origin and $coLie(G)$ is
$Hom_{\mathcal{O}_M}(Lie(G),\mathcal{O}_M)$. Put
$G':=G\times_{\mathcal{O}_M} \mathcal{O}_{K_{F_f,m}}$; then
$Lie(G')\otimes_{\mathcal{O}_M}\Q \cong K_{F_f,m}$ and
$coLie(G')\otimes_{\mathcal{O}_M}\Q \cong K_{F_f,m}$. Consider the
map $\varrho_{\varepsilon}:\plim{n}K_{F_f,n}^*\rightarrow
K_{F_f,m}$ given by the composition
\[\begin{CD}\plim{n}K_{F_f,n}^*
 @>(\;,\;)(\ref{eq10})>> Hom_{\mathcal{O}_M,cont}(F_f(\mathfrak{m}_{f,m}),\mathcal{O}_M) @>\circ exp>> Hom_{\mathcal
{O}_M}(Lie(G'),M)\\
@>Tr>> K_{F_f,m}\cong\Q\otimes_{\mathcal{O}_M}coLie(G')\\
\end{CD} \]\\
where the last isomorphism is Kato's trace pairing
\cite[II,\S2]{kato1} and $exp$ is the exponential map.

Wiles' reciprocity law affirms that $\varrho_{\varepsilon}(a)$ has
an expression in terms of $\dlog$, which can be reformulated by
defining a natural map $\delta_{\varepsilon}$. Here we define
$\varrho_{\varepsilon}$ and $\delta_{\varepsilon}$ in Kato's
generality \cite[\S6]{kato}: this essentially includes Lubin-Tate
formal groups and the generalized Wiles' reciprocity law obtained
in \S3 (\cite[6.1.10]{kato}).

Let $K$ be an $\ell$-dimensional local field with $char(K)=0$ and
$char(k_K)=p>0$. Take $G$ a $p$-divisible group over
$\mathcal{O}_K$ with $dim(G)=1$. Suppose $\Lambda\hookrightarrow
End(G)$ with $\Lambda$ an integral domain over $\Z_p$ which is
free of finite rank as $\Z_p$-module. Suppose $T_pG$ is a free
$\Lambda$-module of rank 1 and fix a generator $\varepsilon$. Let
$K_n/K$ be the field extension corresponding to
$Ker(G_K\rightarrow Aut_{\Lambda}(T_pG/p^n))$. Denote the $p$-adic
$G_K$-representation $T_pG\otimes\Q_p$ by $V_pG$: to it one can
apply Fontaine's theory. By Fontaine's ring $B_{dR,K}$ (here we
just recall that it is a filtered $G_K$-module - see
\cite[\S2]{kato1} for more) one constructs the filtered module
$D_{dR,K}(V_pG):=H^0(K,B_{dR,K}\otimes_{\Q_p}V_pG)$. One has that
$gr^{-1}D_{dR,K}(V_pG)$ is canonically isomorphic to $\Q\otimes
Lie(G)$ and  $dim_K(D_{dR,K}(V_pG))=dim_{\Q_p}(V_pG)$. Then one
defines a map \cite[Proposition(2.3.3),p.118]{kato}
$$F_{DR}:H^{\ell-1}(K,\Z_p(\ell)\otimes
Hom_{\Lambda}(T_pG,\Lambda))\rightarrow
\hat{\Omega}^{\ell-1}_{K}\otimes_{\mathcal{O}_K}coLie(G)$$ where
$\hat{\Omega}^{r}_{K}$ is the $p$-adic completion of
$\Omega^r_{\mathcal{O}_K/\Z}$ tensored by $\Q$.

Kato extends the above $\varrho_{\varepsilon}$ to
\[\begin{CD}\varrho_{\varepsilon}:\plim{n}K_{\ell}^M(K_n)
 @>GSM>> H^{\ell}({K},\Z_p(\ell)\otimes
 Hom_{\Lambda}(T_pG,\Lambda))\\
 @>F_{DR}>> \hat{\Omega}^{\ell-1}_K\otimes_{\mathcal{O}_K}coLie(G)\,.
\end{CD} \]
The recipe to construct $GSM$ is: fix $n$, then compose the Galois
symbol map defined in \S2.4 (but with coefficients in
$(\Z/p^n\Z)(\ell))$ instead of $\Z_p(\ell)$) with
$\cup\varepsilon_n^{-1}:H^{\ell}(K_n,(\Z/p^n\Z)(\ell))\rightarrow
H^{\ell}(K_n,(\Z/p^n\Z)(\ell)\otimes Hom_{\Lambda}(T_pG,\Lambda))$
and take the trace map to have the cohomology group over $K$, and
finally take the limit with respect to $n$.

As for $\delta_{\varepsilon}$: up to some technical details which we
do not reproduce here (see \cite[p.119]{kato}) it is essentially the
map $\dlog: K_{\ell}^M(K_n)\rightarrow\hat{\Omega}^{\ell}_K$ given
by $\{\alpha_1,\ldots,\alpha_{\ell}\}\mapsto
\dlog(\alpha_1)\wedge\cdots\wedge\dlog(\alpha_{\ell})$.

\begin{teo}\label{theo10} One has
$\varrho_{\varepsilon}=(-1)^{\ell-1}\delta_{\varepsilon}$
\cite[Theorem 6.1.9]{kato}.
\end{teo}
\begin{obs2}\label{rk11}{\em Kato obtains a generalized
reciprocity law when $dim(G)=1$, $T_pG$ is $\Lambda$-free with
$rank_{\Lambda}T_pG=\ell \geq1$ and  some technical conditions
from Fontaine's theory are satisfied \cite[Theorem 4.3.4]{kato}.
He also defines $\varrho^s$ (from Fontaine's theory) and
$\delta^s$ (on the $\dlog$ side) with $s\in\N^{\ell}$ as maps
$\plim{n}K_{\ell}^M(Y_n)\rightarrow
\mathcal{O}(Y_m)\otimes_{\mathcal{O}_K}coLie(G)^{\otimes
(\ell+r(s))}$ where $Y_n$ is a scheme representing a functor
related with the $p^n$-torsion points associated to $G$ and
$r(s)\in\N$. }
\end{obs2}
\begin{obs2} {\em Kato's generalized reciprocity laws (theorem
\ref{theo10} and remark \ref{rk11}) did not include a
generalization of Coleman power series. Fukaya \cite{fukaya2}
obtains a $K_2$-analog for the Coleman isomorphism}.
\end{obs2}
\begin{obs2}\label{rk12} {\em In the previous paper \cite[II]{kato1} Kato had obtained a reciprocity law for a
map $\hat{\varrho}_{\varepsilon}$ whose definition involves the
Kummer map (\ref{eq10}) for the tower of fields given by any
classical Lubin-Tate formal group (i.e. $\ell=1$ and $K=M$) and
the dual exponential map (see below) associated to a certain
representation of the formal group (more precisely, to a power of
a Hecke character obtained from a CM elliptic curve). He relates
$\hat{\varrho}_{\varepsilon}$ with a map
$\hat{\delta}_{\varepsilon}$ constructed mainly from $\dlog$. This
reciprocity law is expressed in terms of Coleman power series
\cite[II]{kato1}. Tsuji in \cite[I]{tsuji} extends Kato's
reciprocity law \cite[II]{kato1} to representations coming from
more general Hecke characters}.
\end{obs2}
%Do not have a generalization of Coleman power series.
\end{subsubsection}
\begin{subsubsection}{Global approach}
Now the base field is $M$, in particular $\ell=1$ and
$car(k_M)=p>0$. As explained in \cite[\S3.3]{scholl}, Kato proves
that Wiles' reciprocity (or rather Iwasawa's, since it is the case
$\mathcal{F}_f=\mathbb{G}_m$) is equivalent to the commutativity of
the following diagram (which follows from cohomological properties):
\[\begin{CD}
\Q\otimes\plim{n,Norm}\mathcal{O}_{M(\zeta_{p^n})}^*
@>>{Kummer}> \Q\otimes\plim{n}Hom_{cont}(G_{M(\zeta_{p^n}))},\boldsymbol{\mu}_{p^n}) \\
@V{\dlog}VV   @VV{\zeta_{p^n}\mapsto 1}V \\
\Q\otimes\plim{n,Trace}\Omega^1_{\mathcal{O}_{M(\zeta_{p^n})}/\mathcal{O}_M}
&& \Q\otimes\plim{n} H^1(M(\zeta_{p^n}),\Z/p^n) \\
@V{\dlog\zeta_{p^n}\mapsto 1}VV @VV{cor}V \\
\Q\otimes\plim{t_{-,-}}\mathcal{O}_M[\zeta_{p^n}]/p^n && \Q\otimes\plim{n}H^1(K_m,\Z/p^n) \\
@V{(t_{n,m})_n}VV  @VV{inc}V \\
M(\zeta_{p^m}) @>{\cong}>{\cup\frac{1}{p^m}\log\chi_{cycl}}>
H^1(M(\zeta_{p^m}),\hat{\overline{M}})
\end{CD}\]
where
$t_{n,m}:=\frac{1}{p^{n-m}}Tr_{M(\zeta_{p^n})/M(\zeta_{p^m})}$,
$cor$ is the corestriction map, $\chi_{cycl}$ is the cyclotomic
character as in \S3, $\cup\log\chi_{cycl}$ is the cup product by the
element $\log\chi_{cycl}\in Hom_{cont}(G_M,\Z_p)$, $inc$ is the map
induced by $\plim{n}\Z/p^n=\Z_p\hookrightarrow\hat{\overline{M}}$
and $Hom_{cont}(G_{M(\zeta_{p^n})},\boldsymbol{\mu}_{p^n})\cong
H^1(M(\zeta_{p^n}),\Z/p^n(1))$ by $\zeta_{p^n}\mapsto 1$.

%\[ \begin{CD}
%\Q_p\otimes\liminv\frak o_n^{\ast} @>Kummer>> \Q_p\otimes\liminv H^1(K_n,\mu_{p^n}) \\
%@Vd\log VV  @VVcorV \\
%\Q_p\otimes\liminv\Omega_{\frak o_n/\frak o} && \liminv_{n\geq m}\Q_p\otimes H^1(K_m,\Z/p^n\Z) \\
%@VVV @VVV \\
%K_m @>>> H^1(K_m,.)
%\end{CD} \]

%begin{center}
%\begin{tabular}{lcl}\label{diakato1}
%$\Q\otimes\plim{n,Norm}\mathcal{O}_{M(\zeta_{p^n})}^*$ &
%$\longrightarrow\atop Kummer$& {
%$\Q\otimes\plim{n}Hom_{cont}(G_{M(\zeta_{p^n}))},\boldsymbol{\mu}_{p^n})$}\\
%$\downarrow \dlog  $& & $\downarrow \zeta_{p^n}\mapsto 1 $\\
%$\Q\otimes\plim{n,Trace}\Omega^1_{\mathcal{O}_{M(\zeta_{p^n})}/\mathcal{O}_M}$&&
%$\Q\otimes\plim{n} H^1(M(\zeta_{p^n}),\Z/p^n)$\\
%$\downarrow\dlog\zeta_{p^n}\mapsto 1$&&$\downarrow$ cor\\
%$\Q\otimes\plim{t_{-,-}}\mathcal{O}_M[\zeta_{p^n}]/p^n$&&$\Q\otimes\plim{n}H^1(K_m,\Z/p^n)$\\
%$\downarrow (t_{n,m})_n$&&$\downarrow$ $inc$\\
%$M(\zeta_{p^m})$&$\longrightarrow\atop\cong
%\cup\frac{1}{p^m}\log\chi_{cycl}$&$H^1(M(\zeta_{p^m}),\hat{\overline{M}})$
% \end{tabular}
%\end{center}

Inspired by this cohomological approach, Kato formulates a new
reciprocity law. To state it we need to introduce some more
notation. Take a smooth $\mathcal{O}_M$-scheme $U$, complement of a
divisor $Z$ with relatively normal crossings in a smooth proper
$\mathcal{O}_M$-scheme $X$. We assume there is a theory of Chern
classes for higher Quillen $K$-theory giving functorial
homomorphisms
$$ch_{r}:K_r(U\otimes_{\mathcal{O}_M}\mathcal{O}_{M(\zeta_{p^n})})\rightarrow
H^{r}_{et}(U\otimes_{M}{M(\zeta_{p^n})},\Z_p(r))$$ where the $K_r$'s
are Quillen $K$-theory groups and $H_{et}$ denotes continuous
\`etale cohomology. By the Hochschild-Serre spectral sequence  we
obtain a map
$$HS\circ
ch_{r}:K_r(U\otimes_{\mathcal{O}_M}\mathcal{O}_{M(\zeta_{p^n})})\rightarrow
H^{r}(M(\zeta_{p^n}),H^{r-1}(U\otimes_M\overline{M},\Z_p(r))).$$
Now $V:=H^{r-1}(U\otimes_M\overline{M},\Z_p(r))$ is a $p$-adic
$G_M$-representation and Fontaine's theory applies. By Fontaine's
ring $B_{dR,M}$ (here we just recall that it is a filtered
$G_M$-module - see \cite[\S2]{kato1} for more) one constructs the
filtered $M$-vector space
$D_{dR,M}(V):=(B_{dR,M}\otimes_{\Q_p}V)^{G_M}$. Suppose that $V$ is
de Rham, that is $dim_M(D_{dR,M}(V))=dim_{\Q_p}(V)$. Then the dual
exponential map $exp^*:H^1(M,V)\rightarrow D_{dR,M}^0(V)$ is given
by the composition of
\[\begin{CD}
H^1(M,V)@>inclusion >> H^1(M,\hat{\overline{M}}\otimes_{\Q_p}V)\\
@>\cong>>
H^1(M,\hat{\overline{M}})\otimes_MD_{dR,M}^0(V)@>\cong\cup\log\chi_{cycl}>>D_{dR,M}^0(V),
\end{CD} \]
where $\cup\log\chi_{cycl}$ is the Tate isomorphism
$M=H^0(M,\hat{\overline{M}})\rightarrow H^1(M,\hat{\overline{M}})$
and the first isomorphism is induced by the Hodge-Tate
decomposition $\hat{\overline{M}}\otimes_{\Q_p}V\cong
\oplus_{i\in\Z}\hat{\overline{M}}(-i)\otimes_M gr^iD_{dR,M}(V)$
\cite[p.407]{scholl} (for more on $exp^*$ and how it fits into a
motivic Tamagawa number framework see \cite{blochkato}).

On the $\dlog$ side we need a Chern character into de Rham
cohomology (cohomology of differentials). When $X=Spec(R)$ is
noetherian it exists a map
$\dlog:K_q(Spec(R))\rightarrow\Omega^q_{R/\Z}$ satisfying
$\dlog(a\cup b)=\dlog(a)\wedge\dlog(b)$, and other properties
\cite[p.393]{scholl}. Denote
$H^i_{dR}(U/\mathcal{O}_M):=H^i(X,\Omega^{\cdot}_{X/\mathcal{O}_M}(Z))$
the hypercohomology of the de Rham complex of differentials on $X$
with logarithmic singularities over $Z$: it has a natural
filtration. We recall that
$$D_{dR,M}^0(V)=H^0(X,\Omega^1_{X/\mathcal{O}_M}(\log
Z))\otimes_{\mathcal{O}_M}M=Fil^1H^1_{dR}(U/\mathcal{O}_M)\otimes_{O_M}M.$$

For simplicity assume $M=M_0$, so that
$\Omega^1_{\mathcal{O}_{M(\zeta_{p^n})}/\mathcal{O}_M}$ is generated
by $\dlog\zeta_{p^n}$ and $\mathfrak{d}_{M(\zeta_{p^n})/M}$ is
$p^n(\zeta_{p}-1)^{-1}\mathcal{O}_M[\zeta_{p^n}]$.
\begin{teo}\label{katofin}(Explicit reciprocity law) Take $p>2$. Let $X$ be a
smooth and proper curve over $\mathcal{O}_M$ and take an affine
$U\subset X$ as above. Then the following diagram commutes: {\tiny
\[ \begin{CD}
\plim{n}(K_2(U\otimes\mathcal{O}_{M(\zeta_{p^n})}))\otimes\boldsymbol{\mu}_{p^n}^{\otimes
-1} @>>{HS\circ ch_2}> \plim{n}H^1(M(\zeta_{p^n}),H^1_{et}(U\times_M\overline{M},\boldsymbol{\mu}_{p^n})) \\
@VV{\dlog}V @V{cor}VV \\
\plim{n}H^0(X\otimes\mathcal{O}_{M(\zeta_{p^n})},\Omega^2_{X\otimes\mathcal{O}_{M(\zeta_{p^n})}/\mathcal{O}_M}(\log
Z))(-1) && \plim{n\geq m}H^1(M(\zeta_{p^m}),H^1(U\times_M\overline{M},\boldsymbol{\mu}_{p^n})) \\
@VV{=}V @V{\cong}VV \\
\plim{n}\Omega^1_{\mathcal{O}_{M(\zeta_{p^n})}/\mathcal{O}_M}\otimes
Fil^1(H^1_{dR}(U/\mathcal{O}_M)) && H^1(M(\zeta_{p^m}),H^1(U\times_M\overline{M},\Z_p)(1)) \\
@VVV  @V{exp^*}VV \\
\plim{n}\mathcal{O}_{M(\zeta_{p^n})}/\mathfrak{d}_{M(\zeta_{p^n})/M}\otimes_{\mathcal{O}_M}
Fil^1H^1_{dR}(U/\mathcal{O}_M) @>>{tr}>
M(\zeta_{p^m})\otimes_{\mathcal{O}_M}Fil^1H^1_{dR}(U/\mathcal{O}_M)
\end{CD} \]
} where $tr$ denotes
$(\frac{1}{p^n}tr_{M(\zeta_{p^n})/M(\zeta_{p^m})})_{n\geq m}$.
\end{teo}

\begin{obs2} {\em Take $X=\mathbb{P}^1$ and $U=\mathbb{A}^1-\{0\}=Spec(\mathcal{O}_M[t,t^{-1}])$. Take
$(u_n)_n\in\plim{n}\mathcal{O}_{M(\zeta_{p^n})}^*$ and consider
$\{u_n,t\}\in K_2(U\otimes\mathcal{O}_{M(\zeta_{p^n})})$. Then from
theorem \ref{katofin} one recovers Iwasawa's theorem \ref{theo2}
\cite[Remark p.411]{scholl}. See also remark \ref{rk12}. }
\end{obs2}

\begin{obs2} {\em Theorem \ref{katofin} has been vastly extended
by H\"ark\"onen \cite{har}: he shows that, under some technical
assumptions, the map $exp^*\circ corestriction \circ HS\circ ch_r:
\plim{n}K_r(U\otimes_{\mathcal{O}_M}M(\varepsilon_n))\rightarrow
M(\varepsilon_m)\otimes D_{dR,M}^0(V)$ can be computed in terms of
a $\dlog$ map, for any $r$ with $1\leq r\leq p-2$ (as before
$\varepsilon$ is a fixed generator of $T_pF_f$ for a Lubin-Tate
formal group $F_f$ over $\mathcal O_M$). We refer to \cite{har}
for precise statements.}
\end{obs2}
\end{subsubsection}

\end{subsection}
\begin{subsection}{Rank 1 Drinfeld modules}
Kato's approach inspired the construction of the diagram presented
in this subsection, which is the only new material of this survey
paper. As in the rest of this work, we restrict ourselves to the case of the Carlitz module; however, we remark that theorem \ref{balo} can easily be extended to any rank 1 sign-normalized Drinfeld module, with exactly the same proof ({\em mutatis mutandis}: the changes are the same necessary to pass from the function field results exposed in this paper to the more general statements of \cite{bl}).

Remember that the Galois action on the module of K\"ahler
differentials $\Omega_{\ol_{\Phi,n}/\ol}$ is given by
$\sigma(\alpha d\beta)=\sigma(\alpha)d\sigma(\beta)$, $\sigma\in
Gal(K_{\Phi,n}/K)$, $\alpha, \beta\in\ol_{\Phi,n}$. In particular,
 one has
$$\sigma(d\varepsilon_n)=d(\sigma\varepsilon_n)=d(\Phi_{\chi(\sigma)}(\varepsilon_n))=\chi(\sigma)d\varepsilon_n\,.$$
Besides, $d\varepsilon_n=d\Phi_{\pi^k}(\varepsilon_{n+k})=\pi^kd\varepsilon_{n+k}$.\\

By $\liminv\Omega_{\ol_{\Phi,n}/\ol}$ we denote the limit with
respect to the trace map, defined as usual by
$Tr_m^n(\omega):=\sum_{\sigma\in
Gal(K_{\Phi,n}/K_{\Phi,m})}\sigma(\omega)$ where
$Tr_m^n=Tr_{K_{\Phi,n}/K_{\Phi,m}}$.

\begin{lema} \label{limdiff} Let $(\omega_n)_n\in\liminv\Omega_{\ol_{\Phi,n}/\ol}$
and for each $n$ choose $x_n\in\ol_n$ so that
$\omega_n=x_nd\varepsilon_n$. Then the sequence
$y_n:=\pi^{-n}Tr^n_m(x_n)$ converges to a limit in $K_{\Phi,m}$ for any fixed $m$.
\end{lema}

\begin{proof} To lighten notation put $G_n^r:=Gal(K_{\Phi,r}/K_{\Phi,n})$, $r>n$. The diagram
{\footnotesize \[ \begin{CD}
0 @>>> G_n^r @>>> Gal(K_{\Phi,r}/K) @>>> Gal(K_{\Phi,n}/K) @>>> 0\\
&&   @V{\simeq}VV     @V{\simeq}VV  @V{\simeq}VV \\
0 @>>> (1+\pr^n)/(1+\pr^r) @>>> (A/\pr^r)^* @>>>
(A/\pr^n)^* @>>> 0
\end{CD} \] }
(where vertical maps are induced by $\chi$) shows that $\chi(\sigma)-1 \in\pr^n$ for $\sigma\in G_n^r$.\\
The equality $\omega_n=Tr_n^{n+k}\omega_{n+k}$ can be rewritten as
$$x_n\pi^{k}d\varepsilon_{n+k}=x_nd\varepsilon_n=\sum_{\sigma\in G_n^{n+k}}\sigma(x_{n+k})\chi(\sigma)d\varepsilon_{n+k}\,,$$
that is, $\pi^kx_n=\sum\sigma(x_{n+k})\chi(\sigma)+\delta_{n,n+k}$
for some $\delta_{n,n+k}\in\frak{d}_{K_{\Phi,n+k}/K}$. Let
$$z_{n,n+k}:=\sum_{\sigma\in G_n^{n+k}}\sigma(x_{n+k})(\chi(\sigma)-1)+\delta_{n,n+k}\,.$$
By \cite[Lemma 3]{bl} we know $v(\delta_{n,n+k})>n+k-1$; together
with the observation above, this implies that $v(z_{n,n+k})\geq
n$. It is computed in \cite[cor.4]{bl} that
$v(Tr^n_m(a))>v(a)+n-m-1$: applying it to
$$y_n-y_{n+k}=\frac{1}{\pi^{n+k}}Tr_m^n\left(x_n\pi^k-\sum_{\sigma\in G_n^{n+k}}\sigma(x_{n+k})\right)=Tr_m^n\left(\frac{z_{n,n+k}}{\pi^{n+k}}\right)$$
we see that $v(y_n-y_{n+k})\geq n-(m+k+1)$, proving that the $y_n$'s form a Cauchy
sequence.
\end{proof}

By abuse of notation, we denote the limit in lemma \ref{limdiff}
as
$\displaystyle\lim_{n\rightarrow\infty}\frac{1}{\pi^n}Tr^n_m\!\!\left(\frac{\omega_n}{d\varepsilon_n}\right)\,.$

 Inspired by \cite[\S1.1]{kato} and \cite[Theorem 3.3.15]{scholl}, we construct the following
 diagram:

\[ \begin{CD}
\liminv K_{\Phi,n}^* @>(1)>> Hom_{cont}(\frak m_{\Phi,m},T_{p}\Phi)\\
  @VV(2)V     @V(4)VV\\
\liminv\F_p\frac{d\varepsilon_n}{\varepsilon_n}\oplus\Omega^1_{\ol_{\Phi,n}/\ol_{F_{\pr}}}
@>(3)>> K_{\Phi,m}\simeq Hom_{F_{\pr}}(K_{m,\Phi},F_{\pr})
 \end{CD} \]

\noindent Arrow (1) is the Kummer map: it sends $u=(u_n)_n$ to
$\displaystyle a\rightarrow\lim_{n\rightarrow\infty}(a,u_n)_n.$
(This limit exists in $T_{p}\Phi$:
$$\Phi_{\pi}(a,u_n)_n=((u_n,K_{\Phi,n}^{ab}/K_{\Phi,n})-1)\Phi_{\pi}(\psqrt[n]{a})=(a,u_{n-1})_{n-1}$$
because $\Phi_{\pi}\in\ol\{\tau\}$ commutes with the action of $G_K$; here $\psqrt[n]{a}$ is a root of $\Phi_{\pi}^n(X)=a$.)\\
As for (2), it is just $\dlog:u\mapsto\frac{du}{u}$, extended to
$K_{\Phi,n}^*$ by putting
$\dlog\varepsilon_n^i:=i\frac{d\varepsilon_n}{\varepsilon_n}$, as
described in \cite[\S 4.2.1]{bl}; the limit of differentials is
taken with respect to the trace and
$\frac{d\varepsilon}{\varepsilon}$ denotes the inverse system
$\frac{d\varepsilon_n}{\varepsilon_n}\,$.\\
 The isomorphism
$K_{\Phi,m}\simeq Hom_{F_{\pr}}(K_{\Phi,m},F_{\pr})$ is given by
the trace pairing: $b\in K_{\Phi,m}$ is sent to $a\mapsto
Tr_{K_{\Phi,m}/F_{\pr}}(ab)\,.$ The map (3) is
$$(\omega_n)_n\mapsto\lim_{n\rightarrow\infty}\frac{1}{\pi^n}Tr_{K_{\Phi,n}/K_{\Phi,m}}\!\!
\left(\frac{\omega_n}{d\varepsilon_n}\right)\,.$$ The definition
of (4) needs more explanation. The logarithm $\lambda$ has locally
an inverse $e:\m_{\Phi,n}^{t_n}\rightarrow\m_{\Phi,n}$ for $t_n\gg
0$; this can be extended to $\tilde e:K_{\Phi,n}\rightarrow
F_{\pr}\otimes_{\Phi}\m_{\Phi,n}$ (the tensor product is taken on
$A_{\pr}$, which acts on $\m_{\Phi,n}$ via $\Phi$) by putting
$\tilde e(\pi^iz):=\pi^i\otimes e(z)$ for $i\gg 0$. In order to
define (4), first remember the isomorphism $T_{p}\Phi\simeq
A_{\pr}$, via $a\cdot\varepsilon\leftrightarrow a$; then use
composition with $\tilde e$, $f\mapsto f\circ\tilde e$, to get
$Hom(\m_{\Phi,m},A_{\pr})\rightarrow Hom(K_{\Phi,m},F_{\pr})\,.$

\begin{teo} \label{balo} The diagram above is commutative.
\end{teo}

\begin{proof} Working out definitions, one sees that this is equivalent to part 2 of theorem \ref{theor1}.
More precisely $(3)\circ(2)$ sends $u=(u_n)\in\liminv K_{\Phi,n}^*$ to
the map
$$w\mapsto Tr_{K_{\Phi,m}/K}\left(\frac{w}{\pi^m}\,\dlog Col_u(\varepsilon_m)\right)\,.$$
On the other hand, the image of $u$ under $(4)\circ(1)$ is the
morphism mapping $\pi^iz$, $v(z)\gg0$, to $\pi^ig_u(e(z))$, where
$g_u\in Hom(\mathfrak{m}_{\Phi,m},A_{\pr})$ is uniquely determined
by the condition $g_u(a)\cdot\varepsilon_n=(a,u_n)_n$ for all
$n\geq m$. Recalling that {\footnotesize
$$[e(z),u_n]_n:=Tr_{K_{\Phi,n}/K}\!\!\left(\frac{\lambda(e(z))}{\pi^n}\,
\dlog
Col_u(\varepsilon_n)\right)\cdot\varepsilon_n=Tr_{K_{\Phi,m}/K}\!\!\left(\frac{z}{\pi^m}\,\dlog
Col_u(\varepsilon_m)\right)\cdot\varepsilon_n\,,$$} it is clear
that $(3)\circ(2)=(4)\circ(1)$ iff $(\cdot,\cdot)=[\cdot,\cdot]$.
\end{proof}

\begin{obs2} {\em The similarity of our diagram with the first
diagram of \S 4.2.2 is rather vague. It would be nice to express
(and prove) the reciprocity law in the cohomological setting, as
in \cite[\S3.3]{scholl}; the big problem here is to find a good
analogue of $H^1(G_{\Q_p},\mathbb{C}_p)$ in characteristic $p>0$
(a naive approach cannot work: Y. Taguchi proved that
$H^1(G_K,{\bf C}_{\pr})=0\,$). Recent developments in extending
Fontaine's theory to the equal characteristic case (for a survey
see \cite{hartl}) might be helpful}.
\end{obs2}

\end{subsection}

\end{section}

\vspace{1cm}

Francesc Bars Cortina, Depart. Matem\`atiques, Universitat Aut\`onoma de Barcelona, 08193 Bellaterra. Catalonia. Spain.\\
E-mail: francesc@mat.uab.cat \\

Ignazio Longhi,  Dipartimento di Matematica ``Federigo Enriques", Universit\`a degli Studi di Milano. Via Cesare Saldini 50, 20133 Milan. Italy.\\
E-mail: longhi@mat.unimi.it

\label{lastpageR}

\begin{thebibliography}{100}


 \small\parskip=-0.1cm

    \bibitem{angles} {Angl\`es, B.}:
    \textit{On explicit reciprocity laws for the local Carlitz-Kummer symbols}.
    {J. Number Theory} \textbf{78} (1999), no. 2, 228--252.

    \bibitem{artinhasse}
    {Artin, E., and Hasse, H.}:
    \textit{Die beiden Erg\"anzungss\"atze zum Reziprozit\"atsgesetz der $l^n$-ten \-Po\-tenz\-re\-ste im K\"orper der $l^n$-ten Einheitswurzeln}.
    {Abh. Math. Sem. Univ. Hamburg} \textbf{6} (1928), 146-162.
\bibitem{bl} {Bars, F. and Longhi, I.}:
\textit{Coleman's power series and Wiles' reciprocity for Drinfeld
modules.} {Corrected version at
http://mat.uab.cat/$\sim$francesc/publicac.html}
\bibitem{benois} {Benois, D.}: \textit{P\'eriodes
$p$-adiques et lois de r\'eciprocit\'e explicites}. {J.reine
angew. Math.} \textbf{493} (1997), 115-151.
\bibitem{ben}{Benois, D.}: \textit{On Iwasawa theory of crystalline representations}.
{Duke Math. J.} \textbf{104} (2000), 211-267.
\bibitem{ber}{Berrizbeitia, P.J.}: \textit{An explicit reciprocity theorem in finite extensions of
$\Q_p$}. {Thesis, Massachussets Institute of Technology} (1986).
\bibitem{blochkato} {Bloch, S. and Kato, K.}: \textit{$L$-functions and Tamagawa
numbers of motives}. In {P. Cartier et al. eds.:
 Grothendieck Festschrift Vol. I.}, Birkh\"auser (1990).

\bibitem{cao}{Cao, L.}:
\textit{Explicit reciprocity law for Lubin-Tate formal groups}.
{Acta Math. Sinica, english series} \textbf{22}, No.5 (2006),
1399-1412. \bibitem{coleman} {Coleman, R.}: \textit{Division
values in local fields}. {Inven.Math.} \textbf{53} (1979), 91-116.
\bibitem{deshalit} {de Shalit, E.}: {Iwasawa Theory of Elliptic Curves with
Complex Multiplication}. Persp. in Mathematics \textbf{3},
Acad.Press, Boston, MA 1987.
\bibitem{drinfeld}{Drinfeld, V.G.}: \textit{Elliptic
modules}. {Math.USSR Sbornik} \textbf{23} (1974), No.4, 561-592.
\bibitem{fukaya}{Fukaya, T.}: \textit{Explicit
reciprocity laws for $p$-divisible groups over higher dimensional
local fields}. {J.Reine Angew. Math.} \textbf{531} (2001), 61-119.
%\bibitem{fukaya1}{Fukaya, T.}:\textit{}..Doc.Math.
\bibitem{fukaya2}{Fukaya, T.}: \textit{The theory of
Coleman power series for $K_2$}. {Journal of
Alg.Geometry}\textbf{12} (2003), 1-80.
\bibitem{hartl} {Hartl, U.}:\textit{A dictionary between
fontaine-theory and its analogue in equal characteristic}. {Arxiv
preprint}.
\bibitem{haz}{Hazewinkel, M.}:  {Formal groups and applications}.
Academic Press, New York (1978).
\bibitem{har}{H\"ark\"onen, H.}: \textit{Explicit
reciprocity laws and cohomology of log schemes} Preprint June
2007.
\bibitem{iwasawa}
{Iwasawa, K.}: \textit{On explicit formulas for the norm residue
symbol}. {J. Math. Soc. Japan} \textbf{20} (1968), 151-164.

    \bibitem{katoK}{Kato, K.}:
\textit{A generalization of local class field theory by using
$K$-groups} I {J.Fac.Sci.Univ.Tokyo}\textbf{26} (1979), 303-376;
II {ibid.}\textbf{27} (1980), 603-683; III {ibid.}\textbf{29}
(1982), 31-43.
    \bibitem{kato1} {Kato, K.}: \textit{Lectures on the approach to
Iwasawa theory for Hasse-Weil L-functions via $B_{dR}$}. In
{Arithmetic Algebraic Geometry (Trento, 1991)}, 50-163. LNM 1553.
Springer (1993).
\bibitem{kato}
{Kato, K.}: \textit{Generalized explicit reciprocity laws}. {Adv.
Stud. Contemp. Math. (Pusan)} \textbf{1} (1999), 57-126.
\bibitem{kum}
{Kummer, E.}: \textit{\"Uber die allgemeinen
Reziprozit\"atsgesetze der Potenzreste}. {J. Reine Angew. Math.}
\textbf{56} (1858), 270-279.
\bibitem{kurihara} {Kurihara, M.}: \textit{The
exponential homomorphisms for the Milnor $K$-groups and an
explicit reciprocity law}. {J. Reine Angew. Math.} \textbf{498}
(1998), 201-221.
\bibitem{kur2} {Kurihara, M.}: \textit{Kato's higher local class field
theory}. In {Invitation to higher local fields.}, 53-60. Geometry
and Topology Monographs Vol \textbf{3}, 2000, International Press.
\bibitem{lang}{Lang, S.}: {Cyclotomic Fields I and II}.
     Graduate Texts in Mathematics \textbf{121}. Springer-Verlag, New York, 1990.
\bibitem{par}{Parshin, A.N.}: \textit{Class field theory
and algebraic $K$-theory}. {Uspekhi Mat.Nauk} \textbf{30}, no.1
(1975), 253-254, English transl.:Russian Math.Surv.
\bibitem{scholl}{Scholl, A.J.}: \textit{An introduction
to Kato's Euler systems} In {Galois representations in Arithmetic
Algebraic Geometry (A.J.Scholl and R.L.Taylor eds.)}, 379-460.
London Math.Soc. Lecture Notes \textbf{254}, 1998.
\bibitem{sen}{Sen, S.}: \textit{On explicit reciprocity
laws}. {J.Reine Angew. Math.}\textbf{313} (1980), 1-26.
\bibitem{serre}{Serre, J.P.}: \textit{Local class field
theory}. In {Algebraic Number Fields. J.W.S.Cassels and A.
Fr\"ohlich, eds.}, Academic Press 1967
\bibitem{tate}{Tate, J.}: \textit{$p$-divisible
groups}. In {Proc. of a conf. on local fields.}, 158-183.
Driebergen 1966, Springer (1967).
\bibitem{tsuji}{Tsuji, T.}:
\textit{Explicit reciprocity law and formal moduli for Lubin-Tate
formal groups}. {J.reine angew. Math.}\textbf{569} (2004),
103-173.
\bibitem{vos}
{Vostokov, S.}: \textit{Explicit formulas for the Hilbert symbol}.
In {Invitation to higher local fields.}, 81-89. Geometry and
Topology Monographs Vol \textbf{3}, 2000, International Press.
\bibitem{wiles}
{Wiles, A.}: \textit{Higher explicit reciprocity laws}. {Ann. of
Math.} \textbf{107} (1978), no. 2, 235-254.
\bibitem{zhu}{Zhukov, I.B.}: \textit{Structure theorems for complete fields}
In {Concerning the Hilbert 16th problem. Proceedings of ths
St.Petersburg Mathematical Society.}, 175-192. Amer. Math. Soc.
Transl. (Ser. 2) \textbf{165}, 1995.

\end{thebibliography}
\end{document}